\documentclass{amsart}
\usepackage{amssymb,amsmath,xcolor,graphicx,xspace,colortbl,rotating}
\usepackage{amsfonts}
\usepackage{amsmath}
\usepackage{amssymb,amsmath,xcolor,graphicx,xspace,colortbl,rotating}

\graphicspath{{stationary_anew_prob3_rev3_graphics/}{stationary_anew_prob3_rev3_tcache/}{stationary_anew_prob3_rev3_gcache/}}
\DeclareGraphicsExtensions{.pdf,.eps,.ps,.png,.jpg,.jpeg}
\graphicspath{{stationary_anew_prob3_rev.2tex_graphics/}{stationary_anew_prob3_rev.2tex_tcache/}{stationary_anew_prob3_rev.2tex_gcache/}}
\DeclareGraphicsExtensions{.pdf,.eps,.ps,.png,.jpg,.jpeg}
\newtheorem{theorem}{Theorem}
\theoremstyle{plain}
\newtheorem{acknowledgement}{Acknowledgement}

\newtheorem{corollary}{Corollary}

\newtheorem{definition}{Definition}
\newtheorem{example}{Example}

\newtheorem{lemma}{Lemma}

\newtheorem{proposition}{Proposition}
\newtheorem{remark}{Remark}

\numberwithin{equation}{section}
\input{tcilatex}

\begin{document}
\title[Stationary Markov processes]{On stationary Markov processes with
polynomial conditional moments}
\author{Pawe\l\ J. Szab\l owski}
\address{Department of Mathematics and Information Sciences, \\
Warsaw University of Technology \\
ul. Koszykowa 75, 00-662 Warsaw, Poland}
\email{pawel.szablowski@gmail.com}
\date{October 2013}

\begin{abstract}
We study a class of stationary Markov processes with marginal distributions
identifiable by moments such that every conditional moment of degree say $m$
is a polynomial of degree at most $m\;\text{.}\;$ We show that then under
some additional, natural technical assumption there exists a family of
orthogonal polynomial martingales. More precisely we show that such a family
of processes is completely characterized by the sequence $\{(\alpha _{n}
,p_{n})\}_{n \geq 0}$ where $\alpha _{n}^{ \prime } s$ are some positive
reals while $p_{n}^{ \prime } s$ are some monic orthogonal polynomials.
Paper of Bakry\&Mazet(2003) assures that under some additional mild
technical conditions each such sequence generates some stationary Markov
process with polynomial regression.

We single out two important subclasses of the considered class of Markov
processes. The class of harnesses that we characterize completely. The
second one constitutes of the processes that have independent regression
property and are stationary. Processes with independent regression property
so to say generalize ordinary Ornstein-Uhlenbeck (OU) processes or can also
be understood as time scale transformations of L{\'e}vy processes. We list
several properties of these processes. In particular we show that if these
process are time scale transforms of L{\'e}vy processes then they are not
stationary unless we deal with classical OU- process. Conversely, time scale
transformations of stationary processes with independent regression property
are not L{\'e}vy unless we deal with classical OU process.
\end{abstract}

\maketitle

\section{Introduction}

In this paper we analyze a subclass $S$ of Markov random processes with
polynomial conditional moments that was described in \cite{SzablPoly}.
Namely we confine analysis to Markov processes with polynomial conditional
moments that are additionally stationary. Let $\mathbb{T}$ denote either set
of reals - $\mathbb{\mathbb{R}}$ or $\mathbb{\mathbb{\mathbb{Z}}}$ the set
of integers. By stationary Markov processes we mean those Markov processes $%
X=(X_{t})_{t\in T}$ that have marginal distributions that do not depend on
the time parameter and the property that conditional distributions of say $%
X_{t}$ given $X_{s}$ does depend on $t-s.$

In more detail let $X=(X_{t})_{t\in \mathbb{T}}$ be a real stochastic
process defined on some probability space $(\Omega ,\mathcal{F},P).$ We will
assume that $\forall n\in N,$ $t\in \mathbb{T}$ : $E|X_{t}|^{n}<\infty .$
More precisely we assume that distributions of $X_{t}$ will be identifiable
by their moments. This assumption is slightly stronger assumption than the
existence of all moments. For example it is known that if $\exists \beta
>0:\int \exp (\beta |x|)d\mu (x)<\infty \text{,}$ then measure $\mu $ is
identifiable by its moments. Here $\mu $ denotes distribution of $X_{0}.$ In
fact there exist other conditions assuring this. For details see e.g. \cite%
{Sim98}.

When needed we will assume that for $\forall t\in \mathbb{T}$ : $suppX_{t}$
contains infinite number of points. Sometimes we will omit this assumption
but it will be indicated when. Then, if support of $X_{0}$ consists of $v$
points, the distribution concentrated on these points is identifiable by $v$
orthogonal polynomials including polynomial with the index $0$ equal to $1.$

To fix notation let us denote $\mathcal{F}_{\leq s}=\sigma (X_{r}:r\in
(-\infty ,s]\cap \mathbb{T})$, $\mathcal{F}_{\geq s}=\sigma (X_{r}:r\in
\lbrack s,\infty )\cap \mathbb{T})$ and $\mathcal{F}_{s,u}=\sigma
(X_{r}:r\notin (s,u),r\in \mathbb{T})$.

Moreover let us assume that $\forall n:0<$ $n\leq \nu -1,$ $s_{i}\in \mathbb{%
T},s_{i}\neq s_{j},$ for $i\neq j$ and $i,j=1,\ldots ,n,$ matrix $%
[cov(X_{s_{i}},X_{s_{j}})]_{i,j=1,\ldots ,n}$ is non-singular. Processes
satisfying this assumptions will be called totally linearly dependent
(briefly TLD).

Notice that processes that for every $t\in \mathbb{T}$ are constant i.e. $%
X_{t}=X$ for some random variable $X$ are not TLD.

We will also assume that $\forall m,j:EX_{t}^{m}X_{s}^{j}$ are a continuous
functions of $|t-s|\in \mathbb{T}$ at least at $0$ i.e. for $s$ $=t$. Such
processes will be called mean-square continuous (briefly MSC).

Let us remark that the sequence of independent random variables indexed by
some discrete linearly ordered set are not MSC.

To fix notation let us denote by $\mu (.)$ and by $\eta (.|y,\tau )$
respectively marginal stationary distribution and transition distribution of
our Markov process. That is $P(X_{t}\in A)=\int_{A}\mu (dx)$ and $%
P(X_{t+\tau }\in A|X_{t}=y)=\int_{A}\eta (dx|y,\tau )$. Stationarity of $%
\mathbf{X}$ means thus that $\forall \mathbb{T}\ni \tau \neq 0,B\in B$ 
\begin{equation*}
\mu (B)=\int \eta (B|y,\tau )\mu (dy).
\end{equation*}

By $L_{2}(\mu )$ let us denote the space spanned by real functions that are
square integrable (more precisely equivalence classes) with respect to $\mu $
i.e. 
\begin{equation*}
L_{2}(\mu )=\{f:\mathbb{\mathbb{R}}\longrightarrow \mathbb{\mathbb{R}},\int
|f|^{2}d\mu <\infty \}.
\end{equation*}%
Our assumption on the existence of all moments of $X_{0}$ in terms of $%
L_{2}(\mu )$ implies that there exists a set of orthogonal polynomials that
constitute orthogonal base of this space. Let us denote these polynomials by 
$\{h_{n}\}_{n\geq -1}.$ Additionally let us assume that polynomials $h_{n}$
are orthonormal and $h_{-1}(x)=0,$ $h_{0}(x)=1.$

Notice also that if the support of measure $\mu $ is finite consisting of $v$
points then the space $L_{2}(\mu )$ is $v-$dimensional and there are $v$
orthogonal polynomials $h_{n},$ $n=0,\ldots ,v-1.$

Thus the class of Markov processes that we consider is a class of stochastic
processes that are TLD and MSC and moreover satisfying the following
conditions: $\forall t\in \mathbb{T},n\in N:E(X_{t}^{n})=m_{n}$ and $\forall
n\geq 1,s<t:$

\begin{equation}
E(X_{t}^{n}|\mathcal{F}_{\leq s})=Q_{n}(X_{s},t-s)\text{\ \ }\;\text{a.s.}\;,
\label{p_reg}
\end{equation}%
where $Q_{n}(x,t-s)$ is a polynomial of order not exceeding $n$ in $x\text{.}%
\;$

More precisely let us assume: 
\begin{equation*}
Q_{n}(x,t-s)=\sum_{j=0}^{n}\gamma _{n,j}(t-s)x^{j}.
\end{equation*}

We will call such processes stationary Markov processes with polynomial
regression (briefly SMPR process).

Let us notice that since we assume that the analysed process is MSC we have
for every $n\geq 1$ and $t>s:EX_{t}^{n}X_{s}^{n}\allowbreak =\allowbreak
EQ_{n}(X_{s},t-s)X_{s}^{n}.$ Hence if $Q_{n}$ would not be a polynomial of
order $n$ we would not have required continuity of $EX_{t}^{n}X_{s}^{n}$ (at
least for $t=s)$. Hence in the sequel we will assume that for the SMPR
process polynomials $Q_{n},n\geq 1$, defining conditional moments of order $%
n,$ are exactly of order $n$.

Finally let us underline that from now on all equalities between random
variables will be understand in 'almost sure sense'. However we will drop
abbreviations a.s. for the sake of brevity.

It has to be underlined that neither \cite{SzablPoly} nor this paper are the
first to consider families of orthogonal polynomials that are defined by
some Markov processes and use them to describe completely the marginal and
transitional distributions of the process. It seems that the first were Wim
Schoutens and Jozef L. Teugels who in \cite{SchTeu98} and \cite{Sch2000}
used families of orthogonal polynomials to analyse L{\'e}vy processes.
However it seems that the paper of Julie Lyng Forman and Michael S{\o }%
rensen \cite{ForSor08} contains ideas somewhat the closest to the ideas of
this paper. The point is that they consider diffusion processes having some
number (in many cases finite and depending on several fixed parameters ) of
polynomial conditional moments while we assume that all conditional moments
of the process in consideration are polynomial. Besides we do not assume
that the analysed processes are diffusion consequently having continuous
paths. The only important assumptions that we require are apart from
technical continuity ones, mentioned above, are the stationarity and the
fact that all conditional moments are polynomials of the condition. For
details see next section. Hence the results of \cite{ForSor08} and this
paper are close but neither paper is the generalization of the other.

The paper is organized as follows. The following Section \ref{main} contains
our main results. It consist of three subsections. The first one, Subsection %
\ref{gen} contains general properties of SMPR processes including
construction of orthogonal martingale polynomials, semigroup of transitional
operators and infinitesimal operators. In the second Subsection \ref{harnss}
we consider SMPR processes that additionally are assumed to be harnesses or
quadratic harnesses. We present simple necessary and sufficient conditions
for a SMPR processes to be harnesses and list all quadratic SMPR harnesses
since the list of them is very short, contains only three types of
processes. In Subsection \ref{IR} we analyze a subclass of SMPR processes
that posses independent regression property (generalization of independent
increments property) that is defined in this subsection. We indicate class
of possible marginal distributions and explain relationship of such
processes with L{\'e}vy processes.

Next Section \ref{open} contains some open problems that we were unable to
solve. Finally Section \ref{dow} contains longer proofs.

\section{Stationary processes with polynomial conditional moments\label{main}%
}

\subsection{General properties\label{gen}}

Since conditional expectation of every polynomial of order $n,$ $%
R_{n}(X_{t};\tau )$ (with respect to $\mathcal{F}_{\leq s})$ is a polynomial 
$\hat{R}_{n}(X_{s};\tau ,t-s)$ of the same order there is a natural question
if one can select a polynomial $p_{n}(x;t)$ in such a way that $%
E(p_{n}(X_{t};t)|\mathcal{F}_{\leq s})=p_{n}(X_{s};s)$ i.e. that $%
(p_{n}(X;t),\mathcal{F}_{\leq t})$ is a martingale. In \cite{SzablPoly} it
has been shown that one can always construct a sequence of martingales for a
given SMPR process. We will recall briefly this construction together with
some other notions that were presented there since they turned out to be
useful.

Hence following \cite{SzablPoly} we define for $t\geq s$ polynomials $Q_{n}:$
\begin{equation*}
E(X_{t}^{n}|\mathcal{F}_{\leq s})=Q_{n}(X_{s};t-s).
\end{equation*}%
Using coefficients of the these polynomials (denoted by $\gamma _{n,j}(t-s),$
$j=0,\ldots ,n\text{,}$ $n\geq 0)$ we construct sequence of lower triangular
matrices $\{A_{n}(t)\}_{n\geq 0},$ $t\in \mathbb{T}$ such that $A_{n}$ is a $%
(n+1)\times (n+1)$ matrix with $(0,0)$ entry equal to $1$ and $(i,j)-$th for
entry equal to $\gamma _{i,j}(t).$ Notice that by its construction matrix $%
A_{n}$ is a submatrix of any matrix $A_{k}$ for $k\geq n.$ Matrices $A_{n}$
turned out to be very useful when analyzing processes with polynomial
conditional moments. As pointed out above polynomials $Q_{n}$, $n\geq 1$ are
of exact order $n,$ consequently the matrices $A_{n},$ $n\geq 1$ are
nonsingular. Further by the "tower property" of the conditional expectation
we have: 
\begin{equation}
A_{n}(t-s)A_{n}(u-t)=A_{n}(u-s),  \label{_T}
\end{equation}%
for all $u>t>s\in \mathbb{T}\text{.}\;$ Let us define these matrices for $t<0
$ by the equality:%
\begin{equation*}
A_{n}(-t)=A_{n}(t)^{-1},
\end{equation*}%
for $n\geq 0,$ $t\geq 0.$ From equality (\ref{_T}) and the abovementioned
extended definition, we deduce that matrices do commute and that $%
A_{n}(t)A_{n}(-s)=A_{n}(t-s),$ for all $n\geq 0.$

Moreover we have $A_{n}(0)=I_{n}$ -identity matrix. Hence for every $n\geq 0$
matrices $\{A_{n}(t)\}_{t\in \mathbb{T}}$ constitute an abelian group.

Further following \cite{SzablPoly} these matrices constitute the so called
structural matrices of the process $X\text{.}\;$ Consequently polynomials
defined by 
\begin{equation*}
R_{n}(x,t)=A_{n}(-t)X^{(n)},
\end{equation*}%
where we denoted $(X^{(n)})^{\mathbb{T}}=(1,x,x^{2},\ldots ,x^{n})$
constitute family of polynomials that considered at $X_{t}$ are martingales.
Indeed we have: 
\begin{equation*}
E(R_{n}(X_{t};t)|\mathcal{F}_{\leq
s})=A_{n}(-t)A_{n}(t-s)X_{s}^{(n)}=A_{n}(-s)X_{s}^{(n)}=R_{n}(X_{s};s).
\end{equation*}%
We will add one more technical assumption in order to proceed further
without unnecessary complication.

Before, let us analyze immediate consequences to of above mentioned
properties of matrices $A_{n} (t)$.

\begin{lemma}
\label{diago}There exist a sequence of real nonnegative constants $\{\alpha
_{i}\}_{i\geq 1}$ in the case $\mathbb{T}=\mathbb{\mathbb{R}}$ and $\{\rho
_{i}\}_{i\geq 0}$ such that $|\rho _{i}|<1$ in the case $\mathbb{T}=\mathbb{%
\mathbb{\mathbb{Z}}}\text{,}$ such that $i-th$ diagonal element of the
matrix $A_{n}(t)$ is equal to $\exp (-\alpha _{i}t)$, $t\in \mathbb{\mathbb{R%
}}$ in the first case or $\rho _{i}^{n}$, $n\in \mathbb{\mathbb{\mathbb{Z}}}$
in the second. Moreover for each $i$ we have a vector $v_{i}=(v_{n,i},\ldots
,v_{i,i},0,\ldots ,0)^{T}$ with entries that do not depend on $t$ such that
for $s<t$ 
\begin{equation}
E(\exp (\alpha _{i}t)v_{i}^{T}X_{t}|\mathcal{F}_{\leq s})=\exp (\alpha
_{i}s)v_{i}^{T}X_{s},  \label{m1}
\end{equation}%
when $\mathbb{T}=\mathbb{\mathbb{R}}$ and for $\mathbb{T}=\mathbb{\mathbb{%
\mathbb{Z}}}$ 
\begin{equation}
E(\rho _{i}^{-n}v_{i}^{T}X_{n}|\mathcal{F}_{\leq m})=\rho
_{i}^{-m}v_{i}^{T}X_{m},  \label{m2}
\end{equation}%
for $m<n\text{.}$
\end{lemma}

\begin{proof}
The fact that on the diagonal of the matrix $A_{n}(t)$ must be the described
form, follows the fact in both cases the diagonal elements satisfy
multiplicative form of Cauchy equation in the first case with contiguous
time and with discrete time in the second. For the proof the rest of the
statement we take into account that a lower triangular matrices have their
eigenvalues on the diagonal and the fact that that commuting matrices share
each others eigenspaces. Hence since matrices $A_{n}(t)$ and $A_{n}(s)$ are
commuting there eigenspaces cannot depend on $t$ or $s\text{.}$ Now to
justify properties (\ref{m1}) and (\ref{m2}) we take $v_{i}$ to be the $i-$%
th eigenvector of matrices $A_{n}(t)$ related to eigenvalue $\exp (-\alpha
_{i}t)\text{.}$ We have: $E(\exp (\alpha _{i}t)v_{i}^{T}X_{t}|\mathcal{F}%
_{\leq s})=\exp (\alpha _{i}t)v_{i}^{T}A_{n}(t-s)X_{s}=\exp (\alpha
_{i}t)\exp (-\alpha _{i}(t-s))v_{i}^{T}X_{s}\text{.}$ Similar argument and
calculations are performed in the case of $\mathbb{T}=\mathbb{\mathbb{%
\mathbb{Z}}}\text{.}$
\end{proof}

Now we will add one technical condition that in the light of Lemma \ref%
{diago} will look very natural. Namely we will assume that $\forall n\geq 1$
the matrices $\{A_{n}(t)\}_{n\geq 0,t\in \mathbb{T}}$ are \emph{%
diagonalizable}.

For example symmetric matrices are diagonalizable, matrices with different
eigenvalues are diagonalizable. Moreover one can easily show that if a lower
triangular matrix has all entries below the diagonal not equal to zero than
it is \emph{diagonalizable }iff all its eigenvalues (in this case elements
of the diagonal $\{\exp ( -\alpha _{i} t)\}_{i \geq 0}$) are different.

In our case it means that matrices $\{A_{n}(t)\}_{t\in \mathbb{T}}$ are 
\emph{diagonalizable }iff\emph{\ }their diagonal elements are different. On
the other hand by Lemma \ref{diago} we deduce one can construct a family of
polynomial martingales $\{M_{i}(X_{t},t)=\exp (\alpha
_{i}t)v_{i}^{T}X_{t}\}_{i\geq 1}$of different orders. From the theory of
martingales it follows that $EM_{i}^{2}(X_{t},t)$ (the so called
angle-bracket of the martingale) must be an increasing function of $t\text{.}
$ In our case these functions are proportional to functions $\exp (2\alpha
_{i}t)$ in the case $\mathbb{T}=\mathbb{\mathbb{R}}$ and $\gamma _{i}^{-2n}$
in the case $\mathbb{T}=\mathbb{\mathbb{\mathbb{Z}}}$. Hence in
probabilistic terms \emph{diagonalizability }of $A_{n}(t)$ would mean that
so constructed polynomial martingales of different orders "grow" with
different "speed" which is a very natural condition.

As stated above \emph{diagonalizable} matrices must have the same
eigenspaces so consequently we must have $A_{n}(t)=V_{n}\mathbf{\Lambda }%
_{n}(t)V_{n}^{-1}$ for some matrix $V_{n}$ and the diagonal matrix $\Lambda
_{n}(t).$ It is not difficult to deduce that since matrix $A_{n}(t)$ is to
be lower triangular and matrix $\mathbf{\Lambda }$ is to be diagonal then
matrix $\mathbf{V}$ has to be also lower triangular. Moreover by (\ref{_T})
matrices $\mathbf{\Lambda }_{n}(t)$ satisfy $\mathbf{\Lambda }_{n}(t+s)=%
\mathbf{\Lambda }_{n}(t)\mathbf{\Lambda }_{n}(s)$ for all $t$ and $s$ which
leads (following properties of the Cauchy equation considered in the
multiplicative form for both continuous and discrete forms ) to the
conclusion that \newline
$\Lambda _{n}(t)=diag\{1,\exp (-\alpha _{1}t),\ldots ,\exp (-\alpha _{n}t)\}$
for some reals $\alpha _{i,}$ $i\geq 1$ $t\in \mathbb{\mathbb{R}}$ and $%
\Lambda _{n}(t)=diag\{1,\rho _{1}^{t},\ldots ,\rho _{n}^{t}\},$ for some $%
\rho _{i}\in \mathbb{\mathbb{R}},$ $i\geq 1,$ $t\in \mathbb{\mathbb{\mathbb{Z%
}}}$. For the sake of consistency of notation let us denote $a_{0}=0$ and $%
\rho _{0}=1.$

\begin{remark}
Notice that if the support of stationary measure consists of $v$ points then
there is no sense to consider $n>v-1$ consequently we have only $v-1$
numbers $\alpha _{j},$ $j=1,\ldots ,v-1$ in the continuous case and real
numbers $\rho _{i},$ that matter.
\end{remark}

\begin{remark}
\label{exp}\ Notice that for every $n\geq 1$ in the continuous time case
matrix $A_{n}(t)$ can also be presented in the following form: 
\begin{equation*}
A_{n}(t)=\exp (tW_{n}),
\end{equation*}%
where $W_{n}=V_{n}L_{n}V_{n}^{-1}$ with $L_{n}=diag\{0,-\alpha _{1},\ldots
,-\alpha _{n}\}$ and in the form: 
\begin{equation*}
A_{n}(t)=A_{1n}^{t},
\end{equation*}%
where $A_{1n}=V_{n}^{t}L_{n}V_{n}^{-1},$ and $L_{n}=diag\{1,\rho _{1},\ldots
,\rho _{n}\}$ in the discrete time case.
\end{remark}

\begin{proof}
Follows the fact that $\mathbf{\Lambda }_{n}(t)=\exp (tL_{n})=\sum_{j\geq 0}%
\frac{t^{j}}{j!}L_{n}^{j}.$ Hence $A_{n}(t)=V_{n}\mathbf{\Lambda }%
_{n}(t)V_{n}^{-1}=\sum_{j\geq 0}\frac{t^{j}}{j!}V_{n}L_{n}^{j}V_{n}^{-1}=%
\sum_{j\geq 0}\frac{t^{j}}{j!}W_{n}^{j}=\exp (tW_{n}).$ In the discrete case
we notice that $\Lambda _{n}(t)=L_{n}^{t}.$
\end{proof}

\begin{remark}
Notice that mentioned in the Remark \ref{exp} matrices $W_{n}$ are
'infinitesimal operators' of the strongly continuous subgroup of commuting
operators (in this case finite dimensional) $\{A_{n} (t)\}$ whose existence
is guaranteed by the Hille-Yoshida Theorem.
\end{remark}

Following \cite{SzablPoly} we deduce that sequence of polynomials $%
\{M_{n}(x,t)\}_{n\geq 0}$ defined by the relationship 
\begin{equation}
M_{n}(x,t)=\Lambda _{n}(-t)V_{n}^{-1}X^{(n)},  \label{omart}
\end{equation}%
constitute sequence of polynomial martingales. Indeed following \cite%
{SzablPoly}, we have: \newline
\begin{gather*}
E(M_{n}(X_{t},t)|\mathcal{F}_{\leq s})=\Lambda
_{n}(-t)V_{n}^{-1}E(X_{t}^{(n)}|\mathcal{F}_{\leq s})=\Lambda
_{n}(-t)V_{n}^{-1}A_{n}(t-s)X_{s}^{(n)}= \\
\Lambda _{n}(-t)V_{n}^{-1}V_{n}\Lambda
_{n}(t-s)V_{n}^{-1}X_{s}^{(n)}=\Lambda
_{n}(-s)V_{n}^{-1}X_{s}^{(n)}=M_{n}(X_{s},s).
\end{gather*}%
Now notice that operation $V_{n}^{-1}X^{(n)}$ defines in fact a sequence of
polynomials $\{p_{n}(x)\}_{n\geq 0}.$ Note also that one can chose
polynomials $\{p_{n}\}_{n\geq 1}$ to be monic. These polynomials together
with the sequence $\{\alpha _{n}\}$ define martingales 
\begin{equation}
M_{n}(X_{t},t)=\exp (\alpha _{n}t)p_{n}(X_{t}),  \label{_OM}
\end{equation}%
for the continuous case and 
\begin{equation}
M_{n}(X_{t},t)=p_{n}(X_{t})/\rho _{n}^{t}.  \label{OMD}
\end{equation}%
in the discrete case and generally characterize analyzed Markov process.

\begin{proposition}
\label{pos}For all $\nu >n\geq 1,$ $\alpha _{n}\geq 0$ in the continuous
time case and $\rho _{n}\in (-1,1)$ in the discrete time cases.
\end{proposition}

\begin{proof}
Now from the general theory of martingales it follows that functions $\hat{m}%
_{n}(t)=EM_{n}^{2}(X_{t},t)$ must be a nondecreasing function of $t.$ On the
other hand from stationarity we deduce that $\forall t\in \mathbb{T}$ : $%
Ep_{n}^{2}(X_{t})$ does not depend on $t.$
\end{proof}

\begin{remark}
Notice that if the property of polynomial regression applies only to finite
(say equal to $N)$ first moments of $X_{0}$ then the above presented method
of analysis remains unchanged (due to the fact that matrix $A_{k}(t)$ is a
submatrix of the matrix $A_{n}(t),$ if $N\geq n>k$). Hence we deduce that in
this case there exist $N$ polynomial martingales all having structure as in (%
\ref{_OM}) or (\ref{OMD}).
\end{remark}

From now on we will concentrate more on the continuous parameter case
sporadically pointing out differences with discrete case.

We will write SMPR($\{\alpha _{n},p_{n}\})$ to denote SMPR process with
polynomials $\{p_{n}\}$ and numbers $\{\alpha _{n}\}.$ The numbers $\{\alpha
_{n}\}$ will be called \emph{correlation indices }of a given SMPR.

This representation is unique iff we fix sequence of orthogonal polynomials $%
\{p_{n}\}$ i.e. assuming that either they are orthonormal or are monic.

Note that if support of the stationary measure is finite and consists of $v$
points then the set $\{\alpha _{n} ,p_{n}\}$ characterizing SMPR would be
finite consisting of $v$ points for $n =0 ,\ldots ,v -1.$

\begin{remark}
Any linear combination of martingales $\sum_{j=0}^{n}\beta
_{n,j}M_{j}(X_{t},t),$ $n\geq 1$ with independent on $t$ parameters $\{\beta
_{n,j}\}$ is also a polynomial martingale. However there is only one family
of martingales of the form (\ref{_OM})
\end{remark}

The following proposition lists some of the properties of these martingales
and constants.

\begin{proposition}
\label{easy}i) $\forall n \geq 1 :$ number $\alpha _{n}$ is positive,

ii) $E(\exp (-\alpha _{n}s)p_{n}(X_{s}))|\mathcal{F}_{\geq t})=\exp (-\alpha
_{n}t)p_{n}(X_{t}),$

ii) if $\alpha _{n}\neq \alpha _{m}:EM_{n}(X_{t},t)M_{m}(X_{t},t)=0,$
\end{proposition}

\begin{proof}
i) From the general theory of martingales it follows that $%
EM_{n}^{2}(X_{t},t)=\exp (2\alpha _{n}t)Ep_{n}^{2}(X_{t})$ is an increasing
function of $t$. ii) Follows symmetry in time of the considered process.
iii) Keeping in mind that for $t>s:E(p_{n}(X_{t})|\mathcal{F}_{\leq s})=\exp
(\alpha _{n}(t-s))p_{n}(X_{s})$ and $E(p_{n}(X_{s})|\mathcal{F}_{\geq
t})=\exp (-\alpha _{n}(t-s))p_{n}(X_{t})$, let us calculate $%
Ep_{n}(X_{t})p_{m}(X_{s})$ in two ways. On one hand we have $\exp (\alpha
_{n}(t-s))E(p_{n}(X_{s})p_{m}(X_{s}))$ and on the other $\exp (\alpha
_{m}(t-s))E(p_{n}(X_{t})p_{m}(X_{t})).$ However since we deal with a
stationary process $E(p_{n}(X_{s})p_{m}(X_{s}))=E(p_{n}(X_{t})p_{m}(X_{t})).$
\end{proof}

\begin{definition}
SMPR process such that polynomials $p_{n}$ are orthogonal with respect to
the stationary measure will be called \emph{regular} briefly RSMPR.
\end{definition}

From Proposition \ref{easy} follows the following corollary.

\begin{corollary}
The SMPR($\{\alpha _{n},p_{n}\})$ with correlation indices $\{\alpha _{n}\}$
all different is RSMPR. If the support of the stationary measure of the
considered SMPR is finite consisting of $v$ points then only $\alpha _{j},$ $%
j=1,\ldots ,v-1$ have to be different in order to ensure that the process is
RSMPR.
\end{corollary}

Notice that for the RSMPR we can identify polynomials $p_{n}/\sqrt{\hat{p}%
_{n}}$ , where we denoted $\hat{p}_{n} =E p_{n}^{2} (X_{0})$ and $h_{n}$
(constituting the base of the space $L_{2} (\mu ))$ since both families were
chosen to be orthonormal with respect to the stationary measure $\mu $.
Having polynomials $\{h_{n}\}$ and nonnegative numbers $\{\alpha _{n}\}$ let
us define:

i) operators $U^{t}$ defined on $L_{2} (\mu )$ with values also in $L_{2}
(\mu )$ by the formula $U^{0} =I$ and for $t \geq 0 :$ 
\begin{equation}
L_{2} (\mu ) \ni f = \sum _{n \geq 0}c_{n} h_{n} \longrightarrow U^{t} f =
\sum _{n \geq 0}\exp ( -\alpha _{n} t) c_{n} h_{n} ,  \label{trans}
\end{equation}

\begin{remark}
Notice that operators $U^{t},$ $t\geq 0$ constitute a strongly continuous
semigroup. This is since we obviously have $U^{t}U^{s}=U^{t+s}$ and we have $%
\Vert U^{t}f-f\Vert ^{2}=\sum_{n\geq 1}c_{n}^{2}(1-\exp (-\alpha _{n}t))^{2}%
\underset{t\longrightarrow 0}{\longrightarrow }0.$ If additionally numbers $%
\underset{n\geq 1}{\min }\alpha _{n}>0$ then we even have $\Vert
U^{t}f-f\Vert ^{2}\leq \underset{j\geq 1}{\max }(1-\exp (-\alpha
_{j}t))^{2}\Vert f\Vert ^{2}$, so $\Vert U^{t}-I\Vert \underset{%
t\longrightarrow 0}{\longrightarrow }0$.
\end{remark}

ii) a subset of $L_{2}(\mu )$ defined by: 
\begin{equation*}
D_{A}=\{f:\mathbb{\mathbb{R}}\longrightarrow \mathbb{\mathbb{R}}%
,f=\sum_{n\geq 0}c_{n}h_{n},\sum_{n\geq 0}c_{n}^{2}\alpha _{n}^{2}<\infty
\}\in L_{2}(\mu ),
\end{equation*}

iii) and operator $A$ acting on $D_{A}$ defined by the formula: 
\begin{equation}
D_{A}\ni f=\sum_{n\geq 0}c_{n}h_{n}\longrightarrow Af=-\sum_{n\geq 0}\alpha
_{n}c_{n}h_{n}.  \label{inf}
\end{equation}%
Let us immediately remark that family $\{U^{t}\}_{t\geq 0}$ constitutes (by
its definition) a semigroup of operators on $L_{2}.$ Moreover if numbers $%
\{\alpha _{n}\}$ are such that $\sum_{n\geq 0}\exp (-2\alpha _{n}t)<\infty $
for $t>0$ then operator $U^{t}$ is a Hilbert-Schmidt operator.

We summarize the above mentioned considerations and the results of \cite%
{SzablPoly} adapted to our assumptions in the following theorem.

\begin{theorem}
\label{semigroup}For every RSMPR process $\mathbf{X}$ one can define a
family of polynomials $\{h_{n}\}_{n\geq 1}$ orthonormal with respect to the
marginal, stationary measure and a sequence of positive constants $\{\alpha
_{n}\}_{n\geq 1}$ such that the sequence $(M_{n}(X_{t},t),\mathcal{F}_{\leq
t})_{t\in \mathbb{\mathbb{R}}}$ $n=1,2,3,\ldots $ defined by 
\begin{equation}
M_{n}(X_{t},t)=\exp (\alpha _{n}t)h_{n}(X_{t}),\;\text{{}}\;\;n\geq 1,
\label{MO}
\end{equation}%
constitutes a family of orthogonal martingales.

Family $\{U^{t}\}_{t\geq 0}$ of operators defined by (\ref{trans})
constitutes a strongly continuous semigroup of transition operators of $%
\mathbf{X}$ , i.e. \newline
$\forall f\in L_{2}(\mu ):$ $(U^{t}f)(y)=E(f(X_{t})|X_{0}=y).$

Moreover operator $A$ defined by (\ref{inf}) is the infinitesimal operator
of the semigroup $\{U^{t}\}_{t\geq 0}$ and $D_{A}$ is its domain.
Consequently RSMPR processes are completely characterized by polynomials $%
\{h_{n}\}$ and positive reals $\{\alpha _{n}\}.$

If additionally $\eta <<\mu $ and $\int ((\frac{d\eta }{d\mu }))^{2}d\mu
<\infty ,$ where as above $\mu (dx)$ and $\eta (dx|y,t)$ denote respectively
marginal and transitional measures of $X,$ then 
\begin{equation}
\frac{d\eta }{d\mu }(x|y,t)=\sum_{n\geq 0}\exp (-\alpha
_{n}t)h_{n}(x)h_{n}(y).  \label{gest}
\end{equation}
\end{theorem}

\begin{proof}
As it follows from Proposition \ref{easy} polynomials $\{p_{n}\}$ defined by
(\ref{omart}) must be orthogonal, hence one can select them in such a way
that they are additionally normalized. The fact that operators $(U^{t},$ $%
t\geq 0)$ constitute strongly continuous semigroup was show above. Further
we observe that the set $D_{A}$ contains functions $f$ that have finite
expansions in a Fourier series in polynomials $\{h_{n}\}$ and such functions
form a dense subset of $L_{2}(\mu ).$ Next resolvent $R^{\lambda }$ operator
of the semigroup of operators $U^{t}$ is given by the formula 
\begin{equation*}
(R^{\lambda }f)=\int_{0}^{\infty }\exp (-\lambda t)(U^{t}f)(t)dt=\sum_{n\geq
0}c_{n}h_{n}/(\lambda +\alpha _{n}),
\end{equation*}%
if $f=\sum_{n\geq 0}c_{n}h_{n}.$ This is so since for $f=\sum_{j\geq
0}c_{j}h_{j}$ denoting $f_{n}=\sum_{j\geq 0}^{n}c_{j}h_{j}$ and $R^{\lambda
}f=\sum_{n\geq 0}c_{n}h_{n}/(\lambda +\alpha _{n})$ we have $\Vert
R^{\lambda }f_{n}-R^{\lambda }f\Vert \leq \frac{1}{\lambda }\Vert
f_{n}-f\Vert .$ Besides we have also $\Vert R^{\lambda }f\Vert \leq \frac{1}{%
\lambda }\Vert f\Vert $ so $\Vert (\lambda I-A)^{-1}\Vert \leq 1/\lambda .$
Hence all assumptions of the Hille--Yoshida theorem (compare \cite{wencel})
are fulfilled and we deduce that operator $A$ is an infinitesimal operator
of the semigroup $\{U^{t}\}_{t\geq 0}$. Since infinitesimal operator defines
all the finite dimensional distribution of a Markov process and operator $A$
is defined completely by polynomials $h_{n}$ and numbers $\alpha _{n}$ we
deduce that they characterize RSMPR process.

When $\eta <<\mu $ and $\int ((\frac{d\eta }{d\mu }))^{2}d\mu <\infty $ we
use Theorem 2 of (more precisely formula (3.6)).
\end{proof}

As a corollary we get the result.

\begin{corollary}
If $supp\mu $ is bounded, and if $\forall t \geq 0 :$ 
\begin{equation*}
\sum _{n \geq 0}\exp ( -\alpha _{n} t) \underset{x \in supp\mu }{\sup }\vert
h_{n} (x)\vert <\infty
\end{equation*}%
then the family of transition probabilities is Feller, consequently process $%
\mathbf{X}$ has strong Markov property.
\end{corollary}

\begin{proof}
We use Weierstrass criterion for uniform convergence.
\end{proof}

\begin{remark}
If $\{h_{n}\}$ are the so called Appell polynomials i.e. polynomials
satisfying $h_{n}^{^{\prime }}=nh_{n-1},$ (like e.g. Hermite polynomials)
and numbers $\alpha _{n}=n\alpha $ for some $\alpha >0$ then infinitesimal
operator $A$ is a differential operator.
\end{remark}

\begin{remark}
Notice that the expansion (\ref{gest}) presented in the equivalent form 
\begin{equation*}
d \mu (x) d \mu (y) \sum _{n \geq 0}\exp ( -\alpha _{n} t) h_{n} (x) h_{n}
(y)
\end{equation*}%
is in fact a Lancaster's type expansion of the two dimensional distribution $%
(X_{\tau } ,X_{t +\tau })$ as described in \cite{Lancaster58}, \cite%
{Lancaster63(1)}, \cite{Lancaster63(2)}.
\end{remark}

\begin{remark}
Recently two important papers \cite{Bakry03} and \cite{cuchiero12} appeared.
In those papers the so called polynomial processes are examined. In the
second one the polynomial process is exactly the considered in \cite%
{SzablPoly} processes with polynomial regression. The difference between
those two papers lies in the fact that \cite{SzablPoly} we consider and
exploit polynomial martingales that naturally appear, while in \cite%
{cuchiero12} the other martingales are constructed, not necessarily
polynomial. They are used to analyze certain stochastic differential
equations that appear in financial application. Generally in \cite%
{cuchiero12} only time homogenous Markov processes are analyzed and the
stochastic analysis approach is exploited.

The paper \cite{Bakry03} is closer to the ideas exploited in \cite{SzablPoly}
and in the present paper. Namely in \cite{Bakry03} one starts with Markov
processes whose transition operator has polynomial eigenfunctions and is
given by right hand side of formula (\ref{gest}) with $\exp (-\alpha _{n}t)$
replaced by $c_{n}$ $|c_{n}|\leq 1.$ The authors study conditions for the
sequences $\{c_{n}\}\text{,}$ so that $K(x,dy)=(\sum_{j\geq
0}c_{j}P_{j}(x)P_{j}(x))\mu (dy)$ defines transition operator, where $%
\{P_{j}\}$ are polynomials orthogonal with respect to probability measure $%
\mu .$ From this point of view this paper provides probabilistic model for
the cases considered in \cite{Bakry03}.

It provides also important information on the question of existence of
stationary process with polynomial regression. Namely it provides an answer
to the question when operators $U^{t}$ are positive or another words are
there any restriction on possible $(\alpha _{n},p_{n})$ that characterize
RSMP. In \cite{Bakry03} there is condition given for this namely $%
\sum_{n\geq 0}\exp (-\alpha _{n}t)<\infty .$ Hence paper \cite{Bakry03}
provides important extension on the results of this paper. The questions
considered there concentrate around positivity of so defined operators and
are different from those examined in this paper, where we try to
characterize certain subclasses of the considered class of processes using
available and natural information that characterize RSMP i.e. the sequence $%
(\alpha _{n},p_{n}).$
\end{remark}

\subsection{Harnesses\label{harnss}}

Introduced by Hammersley in \cite{Ham67} harnesses were studied in recent
years by Yor in \cite{Yor05} and Bryc et al. in \cite{BryMaWe07} and the
later papers. We will examine in this subsection which of RSMPR processes
are harnesses. Let us now recall definition of harnesses that was presented
in \cite{SzablPoly}. It is slightly modified original definition that
appeared in \cite{BryMaWe07}.

\begin{definition}
A Markov process $X=(X_{t})_{t\in \mathbb{T}}$ such that $\forall t\in 
\mathbb{T}:E|X_{t}|^{r}<\infty ,$ $r\in N$ is said to be $r-$harness if $%
\forall s<t<u:$ $E(X_{t}^{r}|\mathcal{F}_{s,u})$ is a polynomial of degree $r
$ in $X_{s}$ and $X_{u}.$
\end{definition}

\begin{definition}
$1 -$harness will be called simply harness while the process that is both $r
-$harness for $r =1 ,2$ will be called quadratic harness.
\end{definition}

\begin{remark}
Notice that for a Markov process $\mathbf{X}$ to be a harness is equivalent
that $\forall s,u\geq 0:$ 
\begin{equation}
E(r_{1}(X_{t};t)|\mathcal{F}_{s,u})=a_{L}r_{1}(X_{s};s)+a_{R}r_{1}(X_{u};u))
\label{1h}
\end{equation}%
for some functions $a_{L}=a_{L}(s,t,u)$ and $a_{R}=a_{R}(s,t,u)$ of $s,t,u$,
while to be a quadratic harness the process has to be harness and $\forall
s,u\geq 0;t\in \mathbb{T}:$ 
\begin{gather}
E(r_{2}(X_{t};t)|\mathcal{F}_{s,u})=A_{L}r_{2}(X_{s};s)+A_{R}r_{2}(X_{u};u))
\label{2h} \\
+Br_{1}(X_{s};s)r_{1}(X_{u};u)+C_{L}r_{1}(X_{s};s)+C_{R}r_{1}(X_{u};u)+D%
\text{,}
\end{gather}%
for some $A_{L}=A_{L}(s,t,u),$ $A_{R}=A_{R}(s,t,u),$ $B=B(s,t,u),$ $%
C_{L}=C_{L}(s,t,u)$ and $C_{R}=C_{R}(s,t,u).$ Here $r_{i}(x;t)$ $i=1,2$
denote two monic polynomials of order $i$ such that $Er_{i}(X_{t};t)=0$ and $%
Er_{1}(X_{t};t)r_{2}(X_{t};t)=0$. In this way we avoid assumption that the
marginal distribution has all moments and on the other hand utilize nice
properties of orthogonal polynomials.

Further notice that stationarity of $\mathbf{X}$ implies that in fact $a_{L},
$ $a_{R},$ $A_{L},$ $A_{R},$ $B,$ $C_{L},$ $C_{R}$ do depend only of the
differences i.e. on $t-s$ and $u-t.$
\end{remark}

So first let us study which of the RSMPR processes are harnesses.

\begin{theorem}
\label{harness}A RSMPR process is a harness iff $\forall v>n\geq 2:\alpha
_{n}=n\alpha _{1},$ where $v$ denotes the numbers of points in the support
of the stationary measure and $v=\infty $ if this measure is infinitely
supported.
\end{theorem}

\begin{proof}
Proof is shifted to Section \ref{dow}.
\end{proof}

As an immediate corollary we get the following observation.

\begin{corollary}
A transition operator of RSMPR processes that is a harnesses is
Hilbert--Schmidt.
\end{corollary}

Now let us assume that $\mathbb{T}=\mathbb{\mathbb{R}}$ and define new
process $\mathbf{Y}$ on half line $\mathbb{\mathbb{R}}^{+}$ by the formula: 
\begin{equation}
Y_{\tau }=e^{(\ln \tau )/2}X_{(\ln \tau )/(2\alpha _{1})}=\sqrt{\tau }%
X_{(\ln \tau )/(2\alpha _{1})}.  \label{def}
\end{equation}

\begin{proposition}
\label{aB}Let $\mathbf{X}$ be $a$ harness with $E X_{0} =0$ and let $\mathbf{%
Y}$ be the process defined above. Then:

i) $E(Y_{\tau }-Y_{\sigma })^{2}=\tau -\sigma ,$ for $\tau \geq \sigma \geq
0,$ consequently $EY_{\tau }Y_{\sigma }=\min (\tau ,\sigma ),$

ii) there exist a family of orthogonal monic polynomials $\{h_{n}\}$ such
that for all $n\geq 0,$ $\tau >\sigma \geq 0$ 
\begin{align*}
\tau ^{n/2}E(h_{n}(Y_{\tau }/\sqrt{\tau })|\mathcal{F}_{\leq \sigma })&
=\sigma ^{n/2}h_{n}(Y_{\sigma }/\sqrt{\sigma }), \\
\frac{1}{\sigma ^{n/2}}E(h_{n}(Y_{\sigma }/\sqrt{\sigma })|\mathcal{F}_{\geq
\tau })& =\frac{1}{\tau ^{n/2}}h_{n}(Y_{\tau }/\sqrt{\tau }),
\end{align*}
\end{proposition}

\begin{proof}
i) We have: 
\begin{gather*}
E(Y_{\tau }-Y_{\sigma })^{2}=E(\sqrt{\tau }X_{(\ln \tau )/(2\alpha _{1})}-%
\sqrt{\sigma }X_{(\ln \sigma )/(2\alpha _{1}})^{2}= \\
\tau +\sigma -2\sqrt{\sigma \tau }\exp (-c_{0}(\frac{\ln \tau }{2\alpha _{1}}%
-\frac{\ln \sigma }{2\alpha _{1}})=\tau -\sigma .
\end{gather*}

ii) We obviously also have: 
\begin{equation*}
X_{t}=e^{-\alpha _{1}t}Y_{e^{2\alpha _{1}t}}.
\end{equation*}%
On the other hand by (\ref{_OM}) we have $E(\exp (\alpha _{1}nt)h_{n}(X_{t})|%
\mathcal{F}_{\leq s})=\exp (\alpha _{1}ns)h_{n}(X_{s})$ and $E(\exp (-\alpha
_{1}ns)h_{n}(X_{s})|\mathcal{F}_{\geq t})=\exp (-\alpha _{1}nt)h_{n}(X_{t}).$
Now it remains to change time parameter $t->\tau .$
\end{proof}

As an immediate corollary of the above mentioned Proposition and the L{\'e}%
vy characterization of Brownian motion we have the following observation
concerning continuity of RSMPR harnesses paths.

\begin{theorem}
A RSMPR harness $\mathbf{X}$ with $E X_{0} =E X_{0}^{3} =0$ different from
Ornstein--Uhlenbeck process does not have modification with continuous path.
\end{theorem}

\begin{proof}
First of all notice that general form of $h_{1}(x)$ and $h_{2}(x)$ are
respectively $h_{1}(x)=\beta _{1}(x-\gamma _{10})$ and $h_{2}(x)=\beta
_{2}(x^{2}+\gamma _{21}x+\gamma _{20})$ for some constants $\beta _{1},$ $%
\beta _{2}$, $\gamma _{10},$ $\gamma _{21}$ and $\gamma _{20}.$ Now
condition $EX_{0}=0$ implies that $\gamma _{10}=0$ hence $h_{1}(x)=\beta
_{1}x.$ Further conditions $Eh_{2}(X_{0})=Eh_{1}(X_{0})h_{2}(X_{0})=0$
together with $EX_{0}^{3}=0$ imply that $h_{2}(x)=\beta _{2}(x^{2}-v)$ where 
$v=EX_{0}^{2}.$ Let us scale $\mathbf{X}$ so that $EX_{0}^{2}=1$ and let us
consider process $\mathbf{Y}$ transformed from $\mathbf{X}$ by (\ref{def}).
Then by Proposition \ref{aB},ii) we deduce that both $Y_{\tau }$ and $%
Y_{\tau }^{2}-\tau $ for $\tau \geq 0$ are martingales with respect to
standard $\mathcal{F}_{\leq \tau }$. Now recall L{\'{e}}vy's
characterization of Brownian motion. If $\mathbf{Y}$ had continuous path
then it would have been Brownian motion or process $\mathbf{X}$ an
Ornstein--Uhlenbeck process. Since $\mathbf{X}$ is not OU process it cannot
have continuous path modification.
\end{proof}

As far as quadratic harnesses are concerned we have the following
observations.

\begin{proposition}
\label{simp-q}Let $\mathbf{X}$ be RSMPR be quadratic harness with more than
two different points in the support of the stationary measure. Then:

a) 
\begin{equation*}
\exp (-\alpha _{1}(u-s))BEp_{1}^{2}(X_{0})+D=0,
\end{equation*}

b) 
\begin{align*}
\exp ( -\alpha _{1} (u -s)) B E p_{1}^{3} (X_{0}) +(C_{L} \exp ( -\alpha
_{1} (u -s)) +C_{R}) & =0\text{,} \\
\exp ( -\alpha _{1} (u -s)) B E p_{1}^{3} (X_{0}) +(C_{L} +C_{R} \exp (
-\alpha _{1} (u -s))) & =0.
\end{align*}

c) 
\begin{align*}
\exp (-2\alpha _{1}(t-s))& =A_{L}\exp (-2\alpha _{1}(u-s))+A_{R}+B\exp
(-\alpha _{1}(u-s)), \\
\exp (-2\alpha _{1}(u-t))& =A_{L}+A_{R}\exp (-2\alpha _{1}(u-s))+B\exp
(-\alpha _{1}(u-s)),
\end{align*}%
where $\nu =Eh_{1}^{2}(X_{0})h_{2}(X_{0})$, constants $B$, $D$, $C_{L}$, $%
C_{R}$, are defined in (\ref{2h}) and $p_{i},$ $i=1,2$ are monic versions of
polynomials $h_{i},$ $i=1,2$.
\end{proposition}

\begin{proof}
We will use (\ref{2h}). As polynomials $r_{i}$ let us take monic versions of
polynomials $p_{i},$ $i=1,2.$ Further for simplicity of further calculations
let us assume that polynomials $p_{n}$ are monic. a) We take expectation of
both sides of (\ref{2h}). On the way we use properties of orthogonal
polynomials. b) We multiply both sides of (\ref{2h}) first by $p_{1}(X_{s})$
and then take expectation of both sides secondly we multiply both sides of (%
\ref{2h}) by $p_{1}(X_{u})$ and the take expectation of both sides. As
before we exploit properties of orthogonal polynomials $\{p_{n}\}.$ c) We
multiply both sides of (\ref{2h}) first by $p_{2}(X_{u})$ and then take
expectation of both sides secondly we multiply both sides of (\ref{2h}) by $%
p_{2}(X_{s})$ and the take expectation of both sides. As before we exploit
properties of orthogonal polynomials $\{p_{n}\}.$ On the way we note that $%
Ep_{1}(X_{s})p_{1}(X_{u})p_{2}(X_{u})=\exp (-\alpha
_{1}(u-s))Ep_{1}^{2}(X_{u})p_{2}(X_{u})=\exp (-\alpha
_{1}(u-s))Ep_{2}^{2}(X_{u})$ since we assumed that polynomials $p_{i}$ are
monic we have $p_{1}^{2}(x)=p_{2}(x)+\delta p_{1}(x)+\gamma $ for some $%
\delta $ and $\gamma .$
\end{proof}

Below we will present examples of RSMPR harnesses that are important from
the point of view quadratic harnesses.

\begin{example}[$2-$point symmetric Markov chain]
Let us consider the following symmetric stationary Markov chain.: $X_{0}\in
\{-1,1\},$ 
\begin{equation*}
P(X_{0}=1)=P(X_{0}=-1)=1/2.
\end{equation*}%
For $s<t$ we put 
\begin{eqnarray*}
P(X_{t} &=&1|X_{s}=1)=E(X_{t}=-1|X_{s}=-1)=\frac{1}{2}+\frac{1}{2}\exp
(-\alpha (t-s)), \\
P(X_{t} &=&1|X_{s}=-1)=E(X_{t}=-1|X_{s}=1)=\frac{1}{2}-\frac{1}{2}\exp
(-\alpha (t-s)),
\end{eqnarray*}%
for some $\alpha >0.$ Note that we have $X_{0}^{2k}=1$ and $%
X_{0}^{2k+1}=X_{0},$ $k\geq 0$. Since the state space is finite consisting
of $2$ points there are also only $2$ orthogonal polynomials we see that
this chain is RSMPR. Besides condition for RSMPR given in Proposition \ref%
{harness} is trivially fulfilled hence we deduce that $\mathbf{X}$ is also a
harness. We will call so defined Markov chain a two point symmetric Markov
chain with parameter $\alpha >0,$ briefly $2$SMC$(\alpha ).$
\end{example}

\begin{example}[Ornstein-Uhlenbeck process]
As it is well known it is Gaussian process such that its marginal
distribution are as it is well known is Gaussian say $N(0,\sigma ^{2}).$ $%
cov(X_{t},X_{s})=\sigma ^{2}\exp (-\alpha |t-s|)$ for some $\alpha
_{1}=\alpha >0.$ Hence $X_{t}|X_{s}=y\sim N(\rho y,\sigma ^{2}(1-\rho ^{2})),
$ where we denoted for simplicity $\rho =\rho (t,s)=\exp (-\alpha |t-s|).$
To avoid unnecessary complications let us assume that $\sigma ^{2}=1.$
Visibly transitional distribution is absolutely continuous with respect to
the marginal one. Besides so called probabilistic Hermite polynomials $%
\{H_{n}\}$ are orthogonal with respect to $N(0,1).$ Thus we have: 
\begin{equation*}
E(H_{n}(X_{t})|\mathcal{F}_{\leq s})=\rho ^{n}H_{n}(X_{s})
\end{equation*}%
a.s. Since $\rho ^{n}=\exp (-n\alpha |t-s|)$ we see that $\alpha
_{n}=n\alpha .$ Thus OU process is also harness. Moreover following Poisson
formula we have for all $s\neq t,$ $x,y\in \mathbb{\mathbb{R}}:$ 
\begin{equation}
\exp (-\frac{(x-\rho y)^{2}}{2(1-\rho ^{2})}+\frac{x^{2}}{2})/\sqrt{1-\rho
^{2}}=\sum_{j\geq 0}\rho ^{n}H_{n}(x)H_{n}(y)/n!,  \label{pois}
\end{equation}%
which is a particular case of (\ref{gest}).
\end{example}

\begin{example}[$(\protect\alpha ,q)-$Ornstein-Uhlenbeck process]
It is a generalization of the OU process. This process has appeared first as
side result of more general considerations in \cite{Bo} later also in \cite%
{BryMaWe07}. Its analysis and derivation as a 'continuos time' version of
the discrete time process considered in \cite{Bryc2001S} is given in \cite%
{Szab-OU-W}. Let us assume that $q$ is a parameter $q\in (-1,1).$ In order
not to repeat too much let us remark that marginal distribution of this
process has compact support 
\begin{equation*}
suppX_{0}=[-2/\sqrt{1-q},2/\sqrt{1-q}]
\end{equation*}%
and has density $f_{N}(x|q)$ given by e.g. (2.17) of \cite{Szab-rev} or
(2.7) of \cite{Szab-OU-W}. The polynomials orthogonal with respect to $f_{N}$
are the so called $q-$Hermite polynomials defined by the following 3-term
recurrence: 
\begin{equation*}
xH_{n}(x|q)=H_{n+1}(x|q)+[n]_{q}H_{n-1}(x),
\end{equation*}%
with $H_{-1}(x|q)=0,$ $H_{0}(x|q)=1.$ We denoted here $[0]_{q}=0$, $%
[n]_{q}=1+\ldots +q^{n-1}$ for $n\geq 1$. Besides we have: 
\begin{equation*}
E(H_{n}(X_{t}|q)|\mathcal{F}_{\leq s})=\rho ^{n}H_{n}(X_{s}|q),
\end{equation*}%
where as before we denoted $\rho =\exp (-\alpha (t-s))$ for some $\alpha >0.$
From this formula we deduce that $\alpha _{n}=n\alpha $ so $(\alpha ,q)-$OU
process is a harness.

The transitional distribution has density $f_{CN}(x|y,\rho ,q)$ that is for $%
t>s$ given by (2.9) of \cite{Szab-OU-W}. Moreover the transitional
distribution is absolutely continuous with respect to the stationary one and
we have so called Poisson--Mehler expansion formula 
\begin{equation}
f_{CN}(x|y,\rho ,q)/f_{N}(x|q)=\sum_{n\geq 0}\rho
^{n}H_{n}(x|q)H_{n}(y|q)/[n]_{q}!,  \label{mpois}
\end{equation}%
where $[n]_{q}!=\tprod_{i=1}^{n}[i]_{q},$ with $[0]_{q}!=1$.

Let us remark that the above description and name refers formally to the
case when $\mathbb{T}=\mathbb{\mathbb{R}}$. However in fact the case $%
\mathbb{T}=\mathbb{\mathbb{\mathbb{Z}}}$ in fact has been described by Bryc
in his paper \cite{Bryc2001S} and there the process was called as symmetric
random field with linear regression.
\end{example}

As far as quadratic harnesses that are also RSMPR processes it turns out
that there are surprisingly few of them.

\begin{theorem}
\label{q-Har}A RSMPR process $\mathbf{X}$ with $E X_{0} =E X_{0}^{3} =0$ $E
X_{0}^{2} =1$ is a quadratic harness iff $\forall v \geq n \geq 1 :\alpha
_{n} =n \alpha $ for some $\alpha >0$ where $v$ denotes as before
cardinality of the support of stationary measure and there exist $q \in [ -1
,1]$ such that for

i) $q=-1$ $\mathbf{X}$ is $2$SMC$(\alpha )$ $,$

ii) $q=1$ $\mathbf{X}$ is OU process $cov(X_{s},X_{t})=\exp (-\alpha
|t-s|)EX_{0}^{2}.$

iii) $q \in ( -1 ,1)$ $\mathbf{X}$ is a $(q ,\alpha ) -$OU process.
\end{theorem}

\begin{proof}
Proof is shifted to Section \ref{dow}.
\end{proof}

\begin{remark}
As it follows from \cite{SzablAW}, Thm. 2. conditional density of $%
X_{t}\vert X_{s} =z ,X_{u} =y$ for a $(q ,\alpha ) -$OU process is the so
called Askey--Wilson (AW) density that orthogonalizes the so called AW
polynomials with parameters $z ,\exp ( -\alpha \vert t -s\vert ) ,y ,\exp (
-\alpha \vert u -t\vert )$. Further as shown ibidem ((3.10)) every AW
polynomial of say degree $n$ is a polynomial of the same degree in $z$ and $%
y $ we deduce that $(q ,\alpha ) -$OU process is $r -$harness for every $r
\geq 1.$ Of course similar statements can be made about ordinary OU-process
and $2$SMC$(\alpha )$.
\end{remark}

\subsection{Stationary processes with independent regression property\label%
{IR}}

Now let us consider the subclass of RSMPR processes that have the property
that \newline
$E((X_{t}-E(X_{t}|\mathcal{F}_{\leq s}))^{j}|\mathcal{F}_{\leq s})$ does not
depend on $X_{s}$ for $j=1,\ldots .$ We will call this class a RSMPR
processes with independent regression property (RSMPRIR). We have the
following simple observation.

\begin{proposition}
\label{ININ}Let $\mathbf{X}$ be a RSMPR process with independent regression
property. Assume additionally that $EX_{t}=0$ , $t\in \mathbb{T}$. Then

i) If $\mathbb{T}=\mathbb{\mathbb{R}}$ then $n\geq 0:A_{n}(t)=\exp (tW_{n}),$
where $W_{n}$ is a lower triangular matrix with entries $w_{ij}\allowbreak
=\allowbreak \left\{ 
\begin{array}{ccc}
0 & if & i<j, \\ 
id_{0} & if & i=j, \\ 
(\binom{i}{j})d_{i-j} & if & i>j,%
\end{array}%
\right. ,$ for some constants $d_{0},$ $d_{1},\ldots $ with $d_{0}>0.$ If $%
\mathbb{T}=\mathbb{\mathbb{\mathbb{Z}}}$ then $n\geq 0:A_{n}(k)=A_{n}^{k}(1),
$ $k\in \mathbb{\mathbb{\mathbb{Z}}}$ and $A_{n}(1)$ is a lower triangular
matrix with entries $a_{i,j}\allowbreak =\allowbreak \left\{ 
\begin{array}{ccc}
0 & if & i<j, \\ 
\rho ^{i} & if & i=j, \\ 
(\binom{i}{j})\rho ^{j}d_{i-j} & if & i>j,%
\end{array}%
\right. ,$ where we denoted $\rho =\exp (v_{1})$ and $d_{1},\ldots $ are
some constants.

ii) process $e^{d_{0}t}X_{t}$, in case $t\in \mathbb{\mathbb{R}}$ and $\rho
^{n}X_{n}$, if $n\in N$ have independent increments.
\end{proposition}

\begin{proof}
i) In \cite{SzablPoly} (Proposition 2) it was shown that then coefficients $%
\gamma _{n,j}(t-s)$ are given by the formula: 
\begin{equation*}
\gamma _{n,j}(t-s)=(\binom{n}{j})\gamma _{1,1}^{j}(t-s)\sum_{k=0}^{n-j}(%
\binom{n-j}{k})\gamma _{n-j-k,0}(t-s)\gamma _{1,0}^{k}(t-s)
\end{equation*}%
To simplify further considerations we will assume $EX_{t}=0$ which obviously
results in setting $\gamma _{1,0}(t-s)$ to zero. Further obviously $\gamma
_{0,0}(t)=1.$ Hence for the considered subclass of processes we must have 
\begin{equation}
\gamma _{n,j}(t-s)=(\binom{n}{j})\exp (j(t-s)v_{1})\gamma _{n-j,0}(t-s).
\label{i_inc}
\end{equation}%
Following Remark \ref{exp} we know that $A_{n}(t)=\exp (tW_{n})$ and that if 
$\mathbb{T}=\mathbb{\mathbb{R}}$ we have $W_{n}=.\frac{d}{dt}%
A_{n}(t)|_{t->0^{+}}$. For $\mathbb{T}=\mathbb{\mathbb{\mathbb{Z}}}$
obviously $A_{n}(k)=A_{n}^{k}(1)$ where $A_{n}(1)$ is defined by the
relationship (\ref{i_inc}) with $\rho =\exp (v_{1})$ and $d_{k}$ denoting $%
\gamma _{k,0}(1)$ for brevity. Further notice that we have for $\mathbb{T}=%
\mathbb{\mathbb{R}}$: 
\begin{equation*}
.\frac{d\gamma _{n,j}(t)}{dt}|_{t=0}=\left\{ 
\begin{array}{ccc}
jv_{1} & if & n=j, \\ 
(\binom{n}{j})\frac{d\gamma _{n-j}(t)}{dt}|_{t=0} & if & n>j,%
\end{array}%
\right. 
\end{equation*}%
since $\gamma _{n,0}(0)=0$ for all $n>0.$ Consequently $%
W_{n}=[w_{i,j}]_{i,j=0,1,\ldots ,n}$ where 
\begin{equation*}
w_{ij}=\left\{ 
\begin{array}{ccc}
0 & if & i<j, \\ 
id_{0} & if & i=j, \\ 
(\binom{i}{j})d_{i-j} & if & i>j,%
\end{array}%
\right. 
\end{equation*}%
for some constants $d_{0},\ldots ,d_{n},\ldots $ ii) For $t\in \mathbb{%
\mathbb{R}}:$ Since $E(X_{t}|\mathcal{F}_{\leq s})=\exp (-d_{0}(t-s))X_{s}$
and $E((X_{t}-E(X_{t}|\mathcal{F}_{\leq s}))^{j}|\mathcal{F}_{\leq s})$ is
nonrandom we deduce that $e^{jd_{0}t}E((X_{t}-E(X_{t}|\mathcal{F}_{\leq
s}))^{j}|\mathcal{F}_{\leq s})=E((e^{d_{0}t}X_{t}-e^{d_{0}s}X_{s})^{j}|%
\mathcal{F}_{\leq s})$ is also nonrandom. For $n\in \mathbb{\mathbb{\mathbb{Z%
}}}:$ $E(X_{n}|\mathcal{F}_{\leq m})=\rho ^{-n+m}X_{m},$ for $n>m$ and $%
E((X_{n}-E(X_{n}|\mathcal{F}_{\leq m}))^{j}|\mathcal{F}_{\leq m})$ is
nonrandom. Hence $\rho ^{jn}E((X_{n}-E(X_{n}|\mathcal{F}_{\leq m}))^{j}|%
\mathcal{F}_{\leq m})=E(\rho ^{n}X_{n}-\rho ^{m}X_{m})^{j}|\mathcal{F}_{\leq
m})$ is also nonrandom.
\end{proof}

\begin{remark}
Hence in this case constants $\alpha _{n} =n d_{0}$ so they are different.
Consequently polynomials $\{h_{n}\}$ are orthogonal with respect to the
marginal stationary distribution.

On the other hand the above mentioned form of $w_{ij}$ imposes certain
restrictions on polynomials $\{h_{n}\}.$ Namely we deduce that for fixed $n>0
$ polynomials $h_{1},\ldots ,h_{n}$ depend in general $1+\ldots +n=n(n+1)/2$
coefficients but from the discussed result it follows that these
coefficients are determined by $n$ parameters $d_{i,}$ $i=1,\ldots ,n.$

Besides basing on Theorem \ref{harness} we see that every RSMPR process with
independent regression property is a harness.
\end{remark}

The following Lemma exposes r{\^{o}}le of parameters $d_{i}$ in defining
stationary distribution of $X.$ However to avoid too many unnecessary
complications we will set $d_{0}=1$ (this is equivalent to linear
transformation of time).

For the rest of this subsection let us assume $\mathbb{T}=\mathbb{\mathbb{R}}%
.$

\begin{lemma}
\label{inf_div} i. The process RSMPR process $\mathbf{X}$ with $EX_{t}=0$
has stationary distributions infinitely divisible and its moment generating
function (m.g.f.) $\varphi (y)=Ee^{yX_{0}}$ is given by the relationship: 
\begin{equation}
\varphi (y)=\exp (\sum_{j\geq 2}\delta _{j}\frac{y^{j}}{j!}),  \label{mgfX}
\end{equation}%
where we denoted $\delta _{j}=-d_{j}/j.$ Moreover if $\delta _{2}>0$ then
parameters $\delta _{j}/\delta _{2};$ $j>2$ can be interpreted as $j-2$-th
moments of a certain probability measure $\chi $ identifiable by moments
i.e. 
\begin{equation*}
\varphi (y)=\exp (\delta _{2}\int \frac{\exp (yx)-1-xy}{x^{2}}\chi (dx)).
\end{equation*}

ii) For $t >s$ the moment generating function of $X_{t} -\exp ( -(t -s))
X_{s}$ is equal to 
\begin{equation}
\exp ( \sum _{j \geq 2}\delta _{j} (1 -e^{ -j (t -s)}) y^{j}/j ! .
\label{increment}
\end{equation}
\end{lemma}

\begin{proof}
Proof is Shifted to Section \ref{dow}.
\end{proof}

\begin{remark}
Suppose that $\delta _{2}>0$ then $\delta _{j}/\delta _{2}$ are moments of
the measure $\chi .$ It implies that: 
\begin{equation*}
|\delta _{2k-1}/\delta _{2}|^{1/(2k-1)}\leq (\delta _{2k}/\delta
_{2})^{1/(2k)},(\delta _{2k}/\delta _{2})^{1/(2k)}\leq (\delta
_{2k+2}/\delta _{2})^{1/(2k+2)},
\end{equation*}%
consequently $\delta _{4}/\delta _{2}$ is the variance of $\chi ,$ so if $%
\delta _{4}=0$ then $\delta _{j}=0,$ for $j>2$ and the measure the $\chi $
is degenerated, concentrated at $0.$ If $\delta _{4}>0$, then $\delta
_{2k}>0,$ for $k>2.$
\end{remark}

Let $\mathbf{X}$ be RSMPRIR with $EX_{0}=0$ and moment generating function $%
\exp (\sum_{j\geq 2}\delta _{j}y^{j}/j!).$ We will say that $\mathbf{X}$ is $%
\{\delta _{j}\}-$ RSMPRIR.

Let us now consider process $\mathbf{X}$ that is $\{\delta _{j}\} -$
RSMPRIR, assume that $E X_{0} =0$ and let us consider $\mathbf{Y}$ defined
by process $\mathbf{X}$ according to (\ref{def}).

\begin{proposition}
i) $\mathbf{Y}$ has independent increments, and is a harness,

ii) process $\mathbf{Y}$ is not a L{\'e}vy process unless process $\mathbf{X}
$ is an OU process i.e. polynomials $\{h_{n}\}$ are Hermite polynomials.
More precisely for $\tau >\sigma $ we have 
\begin{align}
E \exp (y Y_{\tau }) & = & \exp ( \sum _{j \geq 2}\delta _{j} \tau ^{j/2}
y^{j}/j !) ,  \label{mgfY} \\
E\exp (y (Y_{\tau } -Y_{\sigma }) =\exp ( \sum _{j \geq 2}\delta _{j} (\tau
^{j/2} -\sigma ^{j/2}) y^{j}/j !) . & &   \notag
\end{align}
\end{proposition}

\begin{proof}
i) follows Proposition \ref{ININ}, ii). However if it was true then $\mathbf{%
Y}$ would be a L{\'e}vy process having infinite number of polynomial
orthogonal martingales. As shown in \cite{SzabLev} this is possible only if $%
\mathbf{Y}$ is a Wiener process. Formulae (\ref{mgfY}) and (\ref{mgfX}) are
direct consequences of (\ref{def}) and (\ref{mgfX}).
\end{proof}

\begin{remark}
For $\mathbf{Y}$ to be a L{\'{e}}vy process we should have $Y_{\tau
}-Y_{\sigma }\sim Y_{\tau -\sigma }$ for $\tau \geq \sigma .$ Which in our
context of processes with all moments existing means that $E(Y_{\tau
}-Y_{\sigma })^{j}$ is a function of $(\tau -\sigma )$ for all $j\geq 1.$
\end{remark}

\begin{remark}
It would be tempting to try to use nice formula (\ref{gest}) to sum kernels
built of polynomials orthogonalizing infinitely divisible measures that
appear as marginal distributions of this class of processes. The things are
however more complicated than it seems at the first sight. Namely recall
that formula (\ref{gest}) is valid if measure defined by the conditional
distribution $\eta (dx|y,t-s)$ of $X_{t}$ given $X_{s}=y$ is absolutely
continuous with respect to the marginal measure of $X_{t}$ i.e. $\mu $. Thus
it seems that considering RSMPR processes $\mathbf{X}$ having as marginal
distribution infinitely divisible absolutely continuous distribution with
unbounded support would yield wanted example. However simple case of shifted
exponential distribution (shifted so that expectation is equal to $0)$
having shifted (in the similar way) Laguerre polynomials as monic orthogonal
polynomials leads to negative conclusion. Namely it turns out that
distribution $\eta $ in this case is a mixture of one point distribution and
an exponential one. This follows simple fact that the moment generating
function of marginal distribution (which is equal to $\exp (-(y+1))$ for $%
y\geq -1$ is equal $\frac{\exp (-y)}{1-y}.$ Similarly for the distribution
of $\rho X_{s}$ where we denoted for simplicity $\rho =\exp (-(t-s))$ moment
generating function is equal to $\frac{\exp (-\rho y)}{1-\rho y}.$ So
according to the formula (\ref{increment}) transitional distribution has
moment generating function equal to 
\begin{equation*}
\exp (-y(1-\rho ))\frac{1-\rho y}{1-y}=\exp (-y(1-\rho ))(\frac{1-\rho }{1-y}%
+\rho ).
\end{equation*}%
Hence we deduce that it is a mixture of one point distribution concentrated
at $-(1-\rho )$ with mass $\rho $ and shifted (by $(1-\rho )$ to the left)
exponential distribution with parameter $1$ weighted $(1-\rho ).$

Similar calculations can be performed in the case Laplace (symmetric
exponential) distribution.

Note also that the above calculations do not apply to Ornstein--Uhlenbeck
process i.e. the case when marginal distribution of $X_{0}$ is Normal. Say $%
N(0,1).$ Then, as elementary calculations show, conditional distribution is
also Normal $N(\rho y,1-\rho ^{2})$ and expansion (\ref{gest}) is in this
case given by (\ref{pois}).
\end{remark}

\begin{remark}
To understand better the difference between RSPMPRIR and L{\'{e}}vy
processes with transformed time let us consider a L{\'{e}}vy process $%
Z=(Z_{t},t\geq 0)$ i.e. we assume that $Z_{0}=0$, $EZ_{t}=0,$ $\forall
0<s<t<u:Z_{u}-Z_{t}$ is independent of $Z_{t}-Z_{s}$ and $Z_{t}-Z_{s}\sim
Z_{t-s}.$ Let us also assume that $E\exp (yZ_{t})=\exp (tQ(y))$ is the
m.g.f. of $Z.$ Assume for simplicity that $EZ_{t}^{2}=t$. Let us consider
new process $X=(X_{\tau };\tau \in \mathbb{\mathbb{R}})$ defined by the
relationship for $\tau \in \mathbb{\mathbb{R}}$ : 
\begin{equation*}
X_{\tau }=e^{-\tau }Z_{\exp (2\tau )}.
\end{equation*}%
So 
\begin{gather*}
E(X_{\tau }|\mathcal{F}_{\leq \sigma })=\exp (-\tau )E[(Z_{\exp (2\tau
))}-Z_{\exp (2\sigma )}+Z_{\exp (2\sigma )})|\mathcal{F}_{\leq \sigma }] \\
\exp (-\tau )Z_{\exp (2\sigma )}=\exp (-(\tau -\sigma ))X_{\sigma },
\end{gather*}%
$EX_{\tau }^{2}=\exp (-2\tau )EZ_{\exp (2\tau )}^{2}=1$ and 
\begin{gather*}
X_{\tau }-\exp (-(\tau -\sigma ))X_{\sigma }=\exp (-\tau )(\exp (\tau
)X_{\tau }-\exp (\sigma )X_{\sigma }) \\
=\exp (-\tau )(Z_{\exp (2\tau )}-Z_{\exp (2\sigma )})
\end{gather*}%
is independent of $Z_{\exp (2\sigma )}$ and consequently on $X_{\sigma }.$
Thus process $\mathbf{X}$ has independent regression property and a constant
variance. It is not however stationary since we have $X_{\tau }-X_{\sigma
}=e^{-\tau }Z_{\exp (2\tau )}-e^{-\sigma }Z_{\exp (2\sigma )}=e^{-\tau
}(Z_{\exp (2\tau )}-Z_{\exp (2\sigma )})+Z_{\exp (2\sigma )}(e^{-\tau
}-e^{-\sigma }).$ So $X_{\tau }-X_{\sigma }$ has m.g.f. equal to the product
of m.g.f. of $e^{-\tau }Z_{\exp (2\tau )-\exp (2\sigma )}$ and m.g.f. of $%
Z_{\exp (2\sigma )}(e^{-\tau }-e^{-\sigma })$. Hence it is equal to 
\begin{equation*}
\exp ((\exp (2\tau )-\exp (2\sigma ))Q(\exp (-\tau ))+\exp (2\sigma )Q(\exp
(-\tau )-\exp (-\sigma ))).
\end{equation*}%
One can easily noticed that this function is not a function $\tau -\sigma $
unless $Q(y)=ay^{2}.$ The case $Q(y)=ay^{2}$ refers to Wiener process
exposing yet again its exceptional r{\^{o}}le among L{\'{e}}vy processes.
\end{remark}

\section{Open Problems\label{open}}

Below we present some interesting open questions:

\begin{enumerate}
\item Do there exist RSMPR processes that have $\alpha _{n} =\alpha _{m}$
for some $n \neq m\;\text{?}\;$ Theoretically they can exist but it would be
interesting to see the example.

\item All known to us examples of RSMPR processes concern harnesses i.e.
cases when $\alpha _{n}=n\alpha _{1};n\geq 1.$ It would be very interesting
to know examples of RSMPR processes with say $\alpha _{n}=O(\sqrt{n}),$ $%
\alpha _{n}=O(n^{2})$ or $\alpha _{n}=1-1/n$ for $n\geq 2.$ \newline
Besides by elementary calculations one can show that if RSMPR process is not
a harness than $E(h_{1}(X_{t})|\mathcal{F}_{s,u})$ for $s<t<u$ cannot be
equal to the sum of two functions from $L_{2}(\mu )$ say $%
l(X_{s},s)+r(X_{u},u).$ What are the examples of $E(h_{1}(X_{t})|\mathcal{F}%
_{s,u})$ in this case?

\item We have shown that every RSMPRIR must be a harness and its stationary
distributions must be infinitely divisible. Is the converse statement true?
That is if a RSMPR harness has infinitely divisible stationary distribution
then does it have independent regression property?

\item Consider RSMPRIR process $X.$ Take $t>s.$ As it follows from the
observation that $X_{t}-\rho X_{s}+\rho X_{s},$ where we denoted $\rho =\exp
(-\alpha |t-s|)$ for some $\alpha .$ Let $g(dz,\rho )$ denote distribution
of $X_{t}-\rho X_{s}$ which is independent of $\rho X_{s}.$ Obviously
conditional distribution of $X_{t}|X_{s}=z$ that is $\eta (dx|z,t-s)$ is
equal to $g(dx-\rho z,\rho ).$ By formula (\ref{increment}) we know m.g.f.
of this distribution namely is equal to $\exp (Q(y)-Q(\rho y))$ if $X_{0}$
that has stationary distribution $\mu $ has m.g.f equal to $\exp (Q(y))$ for
some $Q$ satisfying described in Lemma \ref{inf_div}, i). For which
functions $Q$ is $g<<\mu .$ If there were such functions different from $%
Q(y)=ay^{2}$ (Gaussian case) than we would have universal kernel summation
formula 
\begin{equation*}
\mu (dx)\sum_{j\geq 0}\rho ^{j}h_{j}(x)h_{j}(z)/\hat{h}_{j}=g(dx-\rho z,\rho
),
\end{equation*}%
where $\hat{h}_{j}=Eh_{h}^{2},$ where $h_{j}$ are orthogonal polynomials of
the infinitely divisible measure $\mu $ with m.g.f. $\exp (Q(y))$ and the
m.g.f. of $g$ is $\exp (Q(y)-Q(\rho y)).$ We showed that for the Laguerre
polynomials it is not true but in general it is rather difficult analytic
question with not clear answer.
\end{enumerate}

\section{Proofs\label{dow}}

\begin{proof}[Proof of Theorem \protect\ref{harness}]
As monic polynomials $r_{i}$ we take the monic versions of polynomials $p_{n}
$. So within this proof $\{p_{n}\}$ are assumed to be monic. Having
existence of all moments, the family of orthogonal martingales and time
symmetry of RSMPR processes the definition of $1-$harnesses can be reduced
to the following. The proof will be done for the case $\mathbb{T}=\mathbb{%
\mathbb{R}}.$ We exploit the fact that for RSMPR processes $E(p_{n}(X_{t})|%
\mathcal{F}_{\leq s})=\exp (-\alpha _{n}(t-s))p_{n}(X_{s})$ and $%
E(p_{n}(X_{s})|\mathcal{F}_{\geq t})=\exp (-\alpha _{n}(t-s))p_{n}(X_{t})$
and that $\{p_{n}\}$ are monic orthogonal polynomials satisfying certain
3-term recurrence. The case $\mathbb{T}=\mathbb{\mathbb{\mathbb{Z}}}$ can be
done similarly if one keeps in mind that $E(p_{n}(X_{t})|\mathcal{F}_{\leq
s})=\rho _{n}^{t-s}p_{n}(X_{s})$ and $E(p_{n}(X_{s})|\mathcal{F}_{\geq
t})=\rho _{n}^{t-s}p_{n}(X_{t}),$ for $t>s$.

The RSMPR process is a $1-$harness iff for all $n,m\geq 0:$ 
\begin{equation}
Eh_{m}(X_{s})p_{1}(X_{t})p_{n}(X_{u})=a_{L}Eh_{m}(X_{s})p_{1}(X_{s})p_{n}(X_{u})+a_{R}Eh_{m}(X_{s})p_{1}(X_{u})p_{n}(X_{u}).
\label{1-har}
\end{equation}%
Setting $m=1,$ $n=0$ and then $m=0$ and $n=1$ system of two linear
equations:we obtain 
\begin{gather*}
\exp (-\alpha _{1}(t-s))Eh_{1}^{2}(X_{t})=Eh_{1}^{2}(X_{s})(a_{L}+a_{R}\exp
(-\alpha _{1}(u-s)), \\
\exp (-\alpha _{1}(u-t))Eh_{1}^{2}(X_{t})=Eh_{1}^{2}(X_{u})(a_{L}\exp
(-\alpha _{1}(u-s)+a_{R}).
\end{gather*}%
Since $Eh_{1}^{2}(X_{t})$ does not depend on $t$ we get: 
\begin{gather*}
a_{L}=\frac{\exp (-\alpha _{1}(u-s))(\exp (\alpha _{1}(u-t)-\exp (-\alpha
_{1}(u-t))}{1-\exp (-2\alpha _{1}(u-s))}\text{,} \\
a_{R}=\frac{\exp (-\alpha _{1}(u-s))(\exp (\alpha _{1}(t-s)-\exp (-\alpha
_{1}(t-s))}{1-\exp (-2\alpha _{1}(u-s))}\text{.}
\end{gather*}%
Further taking $m=n-1>1$ we get 
\begin{gather*}
Eh_{n-1}(X_{s})p_{1}(X_{t})p_{n}(X_{u}) \\
=a_{L}Eh_{n-1}(X_{s})p_{1}(X_{s})p_{n}(X_{u})+a_{R}Eh_{n-1}(X_{s})p_{1}(X_{u})p_{n}(X_{u}))%
\text{.}
\end{gather*}%
Now $Eh_{n-1}(X_{s})p_{1}(X_{t})p_{n}(X_{u})=\exp (-\alpha
_{n-1}(t-s)-\alpha _{n}(u-t))Eh_{n}^{2}(X_{t}),$ since polynomials $p_{n}$
are monic and $p_{n-1}p_{1}=p_{n}+ch_{n-1}+dh_{n-2}$ by the fact that
polynomials $p_{n}$ satisfy some 3-term recurrence. Similarly $%
Eh_{n-1}(X_{s})p_{1}(X_{s})p_{n}(X_{u})=\exp (-\alpha
_{n}(u-s))Eh_{n}^{2}(X_{s})$ and $Eh_{n-1}(X_{s})p_{1}(X_{u})p_{n}(X_{u}))=%
\exp (-\alpha _{n-1}(u-s))Eh_{n}^{2}(X_{u})$ Since $%
Eh_{n}^{2}(X_{t})=Eh_{n}^{2}(X_{s})=Eh_{n}^{2}(X_{u})$ by stationarity we
get: 
\begin{equation}
\exp (-\alpha _{n-1}(t-s)-\alpha _{n}(u-t))=a_{L}\exp (-\alpha
_{n}(u-s))+a_{R}\exp (-\alpha _{n-1}(u-s)).  \label{har}
\end{equation}%
To get necessary condition for $\alpha _{n-1}$ and $a_{n}$ we set $%
t-s=u-t=\tau .$ Now our identity becomes: 
\begin{gather*}
\exp (-\tau (\alpha _{n-1}+\alpha _{n}))=(\exp (-2\tau (\alpha _{1}+\alpha
_{n-1}))+\exp (-2\tau (\alpha _{1}+\alpha _{n})) \\
\times \frac{(\exp (\alpha _{1}\tau )-\exp (-\alpha _{1}\tau ))}{1-\exp
(-4\alpha _{1}\tau )} \\
=(\exp (-2\tau (\alpha _{1}+\alpha _{n-1}))+\exp (-2\tau (\alpha _{1}+\alpha
_{n}))\frac{\exp (\alpha _{1}\tau )}{1+\exp (-2\alpha _{1}\tau )}.
\end{gather*}%
Now keeping in mind properties of exponential functions we get system of two
linear equations to be satisfied by $\alpha _{n-1}$ and $\alpha _{n}.$ 
\begin{gather*}
\alpha _{n-1}+\alpha _{n}=2\alpha _{n-1}+\alpha _{1}\text{,} \\
\alpha _{n-1}+\alpha _{n}+\alpha _{1}=2\alpha _{n}\text{,}
\end{gather*}%
which yields $\alpha _{n}=n\alpha _{1}.$ Now let us assume that $\alpha
_{n}=n\alpha _{1}$ and consider (\ref{har}). Now let us assume that $\alpha
_{n}=n\alpha _{1}$ and consider (\ref{har}). Since we deal with RSMPR
process and that $\{p_{n}\}$ are monic orthogonal polynomials we deduce that
identity (\ref{har}) is satisfied for all $n,m\geq 0$ and $s,t,u\in \mathbb{%
\mathbb{R}}$ .
\end{proof}

\begin{proof}[Proof of Theorem \protect\ref{q-Har}.]
First let us notice that all processes mentioned in the assertions i)-iii)
satisfy conditions $EX_{0}=EX_{0}^{3}=0$ $EX_{0}^{2}=1.$ Secondly notice
that all of them are quadratic harnesses. More precisely for $q=-1$ The fact
that such discrete Markov process is a QH follows almost directly the fact
that $X_{0}^{2k}=1,$ $X_{0}^{2k+1}=X_{0}$ for $k\geq 1.$ The Wiener process
was in fact the first example of QH. To get the assertion one has to recall
that Ornstein--Uhlenbeck process is obtained from the Wiener process by
certain continuous time transform that does not change the properties of
conditional expectation. iii) The fact that $q-$Wiener process is a
quadratic harness was noticed by Bryc at all for example in \cite{BryMaWe07}
although the $q-$Wiener process (a process closely related to $(q,\alpha )-$%
OU process) appeared already in \cite{Bo}. Again $q-$OU process is obtained
from the $q-$Wiener process by similar time transformation as the
Ornstein--Uhlenbeck process from the Wiener process.

Hence now let us concentrate on the case of RSMPR process $\mathbf{X}$ with $%
EX_{0}=EX_{0}^{3}=0$ $EX_{0}^{2}=1$ that is a harness i.e. satisfies (\ref%
{2h}). First of all notice that assumption that $Eh_{1}^{3}(X_{0})=0$
implies by Proposition \ref{simp-q}, b) that then $C_{L}=C_{R}=0$ for all $%
s<t<u.$ Secondly notice that functions $A_{L},$ $A_{R},$ $B$ are continuous
functions of $s,t,u$, More over by the symmetry argument $%
A_{L}(s,t,u)=A_{R}(s,t,u)$ if $t-s=u-t.$ Further let us consider discrete
time stationary Markov process $Z_{n}=X_{n\delta },$ $n\in \mathbb{\mathbb{%
\mathbb{Z}}}$ and $\delta >0.$ Now notice that process $\{Z_{n}\}_{n\in 
\mathbb{\mathbb{\mathbb{Z}}}}$ satisfies all assumptions of the formulated
by Bryc in his paper \cite{Bryc2001S}. Another words $\{Z_{n}\}_{n\in 
\mathbb{\mathbb{\mathbb{Z}}}}$ is a stationary random field with linear
regression with coefficients $\rho =\exp (-\alpha _{1}\delta ),$ $%
A=A_{L}(s,s+\delta ,s+2\delta )=A_{R}(s,s+\delta ,s+2\delta ),$ $%
B=B(s,s+\delta ,s+2\delta ),$ $D=C_{L}(s,s+\delta ,s+2\delta )=0.$ Moreover
by Proposition \ref{simp-q},c) we see that $1=B+A(\rho ^{2}+\frac{1}{\rho
^{2}})$ and $D=0.$ Thus we can apply Theorem 3.2 of \cite{Bryc2001S} with
parameter $q$ defined by formula (6.21). This Theorem states that marginal
distribution of $Z_{0}$ is uniquely defined when $q\in \lbrack -1,1].$ In
particular that $q$ cannot depend on $\delta .$ The case $q=-1$ defines
Markov process with two point symmetric marginal distribution. Since we also
have $E(X_{t}|X_{s})=\exp (-\alpha _{1}(t-s))X_{s}$ the process in question
is as described in the assertion. When $q=1$ we Theorem 3.2 of \cite%
{Bryc2001S} states that marginal distribution is Normal $N(0,1).$ If $q\in
(-1,1)$ the marginal distribution is by the same theorem by Bryc uniquely
defined by parameter $q$ with specified family of orthogonal polynomials
which can identified as so called $q-$Hermite. To obtain the $q-$%
Ornstein--Uhlenbeck process one has to refer to the results of \cite%
{Szab-OU-W} where the continuous process $\mathbf{X}$ having property that
all its discrete time versions $X_{n\delta }$ is a stationary random field
as described by Bryc. This process is unique and was described in \cite%
{Szab-OU-W} completely and called $q-$OU process.
\end{proof}

\begin{proof}[Proof od Lemma \protect\ref{inf_div}.]
Notice also that since by the definition of coefficients $\gamma _{n,j}(t-s)$
we have $E(X_{t}^{n}|\mathcal{F}_{\leq s})=\sum_{j=0}^{n}\gamma
_{n,j}(t-s)X_{s}^{j}$ and consequently $m_{n}=\sum_{j=0}^{n}\gamma
_{n,j}(t-s)m_{j}$ where we denoted $m_{n}$ the $n-$th moment of the
stationary distribution of the process $X.$ If we denote $%
m_{n}=(1,m_{1},\ldots ,m_{n})^{T}$ then we see that vector $m_{n}$ is the
eigenvector of the matrix $A_{n}(t)$ referring to eigenvalue that is equal
to $1.$ Further taking into account the fact that $A_{n}(t)=\exp (tW_{n})$
we deduce that vector $m_{n}$ satisfies for every $n\geq 1$ equation; 
\begin{equation*}
W_{n}m_{n}=\mathbf{0}_{n},
\end{equation*}%
where $\mathbf{0}_{n}=(0,\ldots ,0)^{T}\in \mathbb{\mathbb{R}}^{n+1}.$ Let $%
\varphi (y)=\sum_{j\geq 0}m_{j}y^{j}/j!=Ee^{yX_{0}},$ $D(y)=\sum_{j\geq
0}d_{j}y^{j}/j!$ be generating functions of the sequences $\{m_{n}\}$ and $%
\{d_{n}\}$ respectively. Keeping in mind Proposition \ref{ININ}, we obtain 
\begin{gather*}
1=1+\sum_{j\geq 1}0y^{j}/j!1=1+\sum_{j\geq 1}\frac{y^{j}}{j!}%
\sum_{k=0}^{j}w_{i,k}m_{k} \\
=1+\sum_{j\geq 1}\frac{y^{j}}{j!}((j-1)d_{0}m_{j}+\sum_{k=0}^{j}(\binom{j}{k}%
)d_{j-k}m_{k})
\end{gather*}%
\begin{gather*}
=1+d_{0}\sum_{j\geq 2}\frac{y^{j}(j-1)}{j!}m_{j}+\sum_{j\geq 1}\sum_{k=0}^{j}%
\frac{y^{j-k}d_{j-k}}{(j-k)!}\frac{y^{k}m_{k}}{k!} \\
=d_{0}\sum_{j\geq 2}\frac{y^{j}(j-1)}{j!}m_{j}+\sum_{j\geq 0}\sum_{k=0}^{j}%
\frac{y^{j-k}d_{j-k}}{(j-k)!}\frac{y^{k}m_{k}}{k!}
\end{gather*}%
\begin{gather*}
=d_{0}\sum_{j\geq 2}\frac{y^{j}(j-1)}{j!}m_{j}+\sum_{k\geq 0}\frac{y^{k}m_{k}%
}{k!}\sum_{j\geq k}\frac{y^{j-k}d_{j-k}}{(j-k)!} \\
=d_{0}\sum_{j\geq 2}\frac{y^{j}(j-1)}{j!}m_{j}+\varphi (y)D(y)
\end{gather*}%
\begin{gather*}
=d_{0}\sum_{j\geq 1}\frac{jy^{j}}{j!}m_{j}-d_{0}(\varphi (y)-1)+\varphi
(y)D(y) \\
=d_{0}y\varphi ^{\prime }(y)+\varphi (y)(D(y)-d_{0})+d_{0}.
\end{gather*}%
Now let us assume for simplicity that $d_{0}=1$ (this is a matter of
rescaling time parameter). Hence we end up with differential equation: 
\begin{equation*}
-y\frac{\varphi ^{\prime }(y)}{\varphi (y)}=D(y)-1,
\end{equation*}%
with initial condition $\varphi (0)=1.$ Consequently $\varphi (y)=\exp
(\int_{0}^{y}(\frac{1-D(x)}{x}dx)=$ \newline
$\exp (\sum_{j\geq 2}\delta _{j}\frac{y^{j}}{j!}),$ where we denoted $\delta
_{j}=-d_{j}/j,$ $j\geq 2.$ Thus $X_{0}$ has m.g.f. of the form $\varphi
(y)=\exp (D(y)),$ where $D(y)$ is an analytic function with $D(0)=D^{\prime
}(0)=0.$ Notice that if $X$ and $Y$ are two independent random variables
with m.g.f. of the form $\exp (D_{1}(y))$ and $\exp (D_{2}(y))$ respectively
then $X+Y$ has m.g.f of the same type of the form $\exp (D_{1}(y)+D_{2}(y)).$
This remark proves that $X_{0}$ has infinitely divisible law. One can also
refer to the results of \cite{SzabLev}, where similar formula for the moment
generating function of marginal distribution was obtained. Following way of
reasoning presented there we deduce that stationary distribution of $X_{0}$
is infinitely divisible and the we know that by assumptions the variance of $%
X_{0}$ exists. Then if this variance is (one can easily deduce that it must
be equal to $-d_{2}/2=\delta _{2})$ equal to zero then the stationary
distribution is degenerate and equal to zero. If however $\delta _{2}>0$
then we can refer to the Kolmogorov's representation of the characteristic
function of the infinitely divisible law and following argument presented in 
\cite{SzabLev}, Remark 3 deduce that $\delta _{j}/\delta _{2}$ is the $j-2$%
th moment of the L{\'{e}}vy measure defining infinitely divisible
distribution $X_{0}.$ Measure $\chi $ is the L{\'{e}}vy measure of the law
of $X_{0}$.
\end{proof}

\begin{acknowledgement}
The author is very grateful to unknown referee for pointing out numerous
misprints and nonsense as well as forcing author to give probabilistic
interpretation of some technical assumptions.
\end{acknowledgement}

\end{document}